\documentclass[12pt]{amsart}
\usepackage{amssymb,verbatim,enumerate,ifthen}
\usepackage[mathscr]{eucal}
\usepackage[utf8]{inputenc}
\usepackage[T1]{fontenc}
\usepackage{algpseudocode}
\usepackage{tikz}
\usetikzlibrary{automata,positioning,arrows.meta,shapes}\makeatletter
\@namedef{subjclassname@2010}{%
 \textup{2010} Mathematics Subject Classification}
\makeatother

\newenvironment{psmallmatrix}
  {\left(\begin{smallmatrix}}
  {\end{smallmatrix}\right)}

\newtheorem{thm}{Theorem}[section]

\newtheorem{prop}{Proposition}[section]
\newtheorem{lem}{Lemma}[section]

\newtheorem{rem}{Remark}

\theoremstyle{definition}
\newtheorem{exa}{Example}[section]

\numberwithin{equation}{section}

\newcommand\nolabel[1]{\nonumber}

\newcommand\R{\mathbb{R}}
\renewcommand\P{\mathscr{P}}
\newcommand\N{\mathbb{N}}
\newcommand\Q{\mathbb{Q}}
\newcommand\Z{\mathbb{Z}}
\newcommand\id{\mathrm{id}}

\newcommand{\QA}[1]{\mathscr{A}^{[#1]}}
\newcommand{\D}{\Delta}

\newcommand{\norma}[1]{\left\| #1 \right\| }
\newcommand{\abs}[1]{\left| #1 \right| }

\DeclareMathOperator{\sign}{sign}

\DeclareMathOperator{\cl}{cl}

\DeclareMathOperator*{\interior}{int}
\newcommand{\Lv}{\mathscr{L}}
\newcommand{\M}{\mathscr{M}}
\newcommand{\G}{\mathscr{G}}
\renewcommand{\D}{\mathscr{D}}
\newcommand{\B}{\mathscr{B}}
\newcommand{\TYPE}[1]{\mathcal{E}_{#1}}
\newcommand{\TYPEP}[1]{\mathcal{E}_{#1}^+}

\numberwithin{equation}{section}
\def\eq#1{{\rm(\ref{#1})}}
\def\Eq#1#2{\ifthenelse{\equal{#1}{*}}
  {\begin{equation*}\begin{aligned}[]#2\end{aligned}\end{equation*}}
  {\begin{equation}\begin{aligned}[]\label{#1}#2\end{aligned}\end{equation}}}

\author[P. Pasteczka]{Pawe\l{} Pasteczka}
\address{Institute of Mathematics \\ Pedagogical University of Krak\'ow \\ Podchor\k{a}\.zych str. 2, 30-084 Krak\'ow, Poland}
\email{pawel.pasteczka@up.krakow.pl}

    

\title{Online premeans and their computation complexity}
\begin{document}
\begin{abstract}
 We extend some approach to a family of symmetric means (i.e. symmetric functions $\mathscr{M} \colon \bigcup_{n=1}^\infty I^n \to I$ with $\min\le \mathscr{M}\le \max$; $I$ is an interval). 
Namely, it is known that every symmetric mean can be written in a form $\mathscr{M}(x_1,\dots,x_n):=F(f(x_1)+\cdots+f(x_n))$, where 
 $f \colon I \to G$ and $F \colon G \to I$ ($G$ is a commutative semigroup).  

For $G=\mathbb{R}^k$ or $G=\mathbb{R}^k \times \mathbb{Z}$ ($k \in \mathbb{N}$) and continuous functions $f$ and $F$ we obtain two series of families (depending on $k$). It can be treated as a measure of complexity in a family of means (this idea is inspired by theory of regular languages and algorithmics).
 
As a result we characterize celebrated families of quasi-arithmetic means ($G=\mathbb{R}\times \mathbb{Z}$) and Bajraktarevi\'c means ($G=\mathbb{R}^2$ under some additional assumptions). Moreover, we establish certain estimations of complexity for several other classical families.
 
\end{abstract}

\maketitle

\section{Introduction}

In most cases means are defined using explicit formulas. In fact there are only few 
general approaches to this topic. One of the most famous are so-called Chisini means (or level-surface means) \cite{Chi29} which allows to express all reflexive means in a unified form.

We provide alternative way of defining means based on some ideas emerging from the theory of regular languages. Our results bind two different scopes which, to the best of author's knowledge, were not considered together earlier. 
Due to this fact introduction is divided into few parts which are devoted to means (sec. 1.1 and 1.2), regular languages (sec. 1.3), and some algorithmic approach to solving problems (sec. 1.4).

\subsection{Means and premeans} We call $\M \colon \bigcup_{n=1}^\infty I^n \to I$ to be a \emph{mean} (or a \emph{mean on $I$} to emphasize its domain) if $\min(a) \le \M(a) \le  \max(a)$  for all $a \in \bigcup_{n=1}^\infty I^n$ (it is often called simply \emph{mean property}). 
It implies that $\M$ is \emph{reflexive}, i.e.
\Eq{*}{
\M(\underbrace{v,\dots,v}_{k\text{ times}})=v\qquad (v \in I,\,k \in \N)\:.
}

Reflexive functions are used to be called \emph{premeans} (see Matkowski \cite{Mat06b}).

Sometimes we restrict a domain of $\M$ to $I^n$ for some $n \in \N$ and we say about \emph{$n$-variable} mean (which is formally not a mean from the point of view of the previous definition). In particular, there are many $2$-variable means that have no obvious extensions to general $n$-tuples, for example Cauchy or Heronian means (see \cite[section~VI.2]{Bul03} for details and more examples).

Let us also recall \emph{level surface means} (see \cite[section~VI.4.1]{Bul03} and references therein). Let $F\colon I^n \to \R$ then the \emph{$F$-level mean} of $a=(a_1\dots,a_n) \in I^n$ equals $\mu$, where
\Eq{*}{
F(\mu, ...,\mu) = F(a_1....a_n),
}
provided $F$ is such that $\mu$ is uniquely determined for all $a \in I^n$.

Note that if $F$ is a premean then $F$-level mean equals $F$. Therefore every premean restricted to $I^n$ ($n \in \N$) is a level surface mean. 
Due to this fact level surface means are considered as the way of thinking or the way of expressing means rather than the family. Online premeans are in the same flavour (compare Remark~\ref{rem:premean}) but, conversely to level surface mean, we redefine the family of all symmetric premeans.

Now we recall few properties of premeans. We say that $\M \colon \bigcup_{n=1}^\infty I^n \to I$ is \emph{continuous} (\emph{symmetric}) if for all $n \in \N$ its restriction $\M|_{I^n}$ is continuous (symmetric).
For $I=\R_+$ we can define \emph{homogeneity} in the same way. Premean $\M$ is called \emph{repetition invariant} if, for all $n,m\in\N$
and $(x_1,\dots,x_n)\in I^n$, the following identity is satisfied
\Eq{*}{
  \M(\underbrace{x_1,\dots,x_1}_{m\text{-times}},\dots,\underbrace{x_n,\dots,x_n}_{m\text{-times}})
   =\M(x_1,\dots,x_n).
}
This property was introduced, in a mean setting, by P\'ales-Pasteczka \cite{PalPas16}.

Finally, element $e \in I$ is called \emph{negligible element of $\M$} if for every vector $a=(a_1,\dots,a_n)$ ($n \ge 2$) such that $a_s=e$ for some $s \in \{1,\dots,n\}$ we have $\M(a)=\M((a_{i})_{i \in \{1,\dots,n\}\setminus \{s\}})$. In the other words element $e$ does not affect to a value of mean unless $a$ is a constant vector having all entries equal to $e$. Obviously each mean has at most one negligible element.

\subsection{Selected families of means} In this section we introduce four closely related families of means.

Define for $p\in\R$ the $p$th power (or H\"older) mean of the positive numbers $x_1,\dots,x_n$ by
\Eq{*}{
  \P_p(x_1,\dots,x_n)
   :=\left\{\begin{array}{ll}
    \Big(\dfrac{x_1^p+\cdots+x_n^p}{n}\Big)^{\frac{1}{p}} 
      &\mbox{if }p\neq0, \\[3mm]
      \sqrt[n]{x_1\cdots x_n}\qquad
      &\mbox{if }p=0.
    \end{array}\right.
}

Let us now indroduce a family which was defined in 1920s/30s \cite{Kno28,Kol30,Nag30,Def31}. Let $I\subseteq\R$ be an interval and $f\colon I\to\R$ be a continuous strictly monotonic function then the quasi-arithmetic 
mean $\QA{f}:\bigcup_{n=1}^\infty I^{n}\to I$ is defined by
\Eq{*}{
  \QA{f}(x_1,\dots,x_n)
   :=f^{-1}\bigg(\frac{f(x_1)+\cdots+f(x_n)}{n}\bigg),
   \qquad x_1,\dots,x_n\in I.
}
By taking $f$ as a power function or a logarithmic function on $I=\R_+$, the resulting quasi-arithmetic mean is a power 
mean. 


Another extension of power means was proposed in 1938 by Gini \cite{Gin38}. For $p,q\in\R$, the \emph{Gini mean} $\G_{p,q}$ of 
the variables $x_1,\dots,x_n>0$ is defined as follows:
\Eq{GM}{
  \G_{p,q}(x_1,\dots,x_n)
   :=\left\{\begin{array}{ll}
    \left(\dfrac{x_1^p+\cdots+x_n^p}
           {x_1^q+\cdots+x_n^q}\right)^{\frac{1}{p-q}} 
      &\mbox{if }p\neq q, \\[4mm]
     \exp\left(\dfrac{x_1^p\ln(x_1)+\cdots+x_n^p\ln(x_n)}
           {x_1^p+\cdots+x_n^p}\right) \quad
      &\mbox{if }p=q.
    \end{array}\right.
}
Clearly, in the particular case $q=0$, the mean $\G_{p,q}$ reduces to the $p$th power mean $\P_p$. It is also 
obvious 
that $\G_{p,q}=\G_{q,p}$. 

A common generalization of quasi-arithmetic means and Gini means can be obtained in terms of 
two arbitrary real functions. These idea was realized by Bajraktarevi\'c \cite{Baj58}, \cite{Baj69} in 1958. Let 
$I\subseteq\R$ be an interval and let $f,g:I\to\R$ be continuous functions such that $g$ is positive and $f/g$ is 
strictly monotone. Define the \emph{Bajraktarevi\'c mean} $\B_{f,g}:\bigcup_{n=1}^\infty I^{n}\to I$ by
\Eq{BM}{
  \B_{f,g}(x_1,\dots,x_n)
   :=\Big(\frac{f}{g}\Big)^{-1}\bigg(\frac{f(x_1)+\cdots+f(x_n)}
                      {g(x_1)+\cdots+g(x_n)}\bigg),
                      \qquad x_1,\dots,x_n\in I.
}
One can check that $\B_{f,g}$ is a mean on $I$. In the particular case $g\equiv1$, the mean $\B_{f,g}$ reduces to 
$\QA{f}$, that is, the class of Bajraktarevi\'c means is more general than that of the quasi-arithmetic means. By putting 
$(f,g)=(x^p,x^q)$ or $(f,g)=(x^p \ln(x),x^p)$ we can see that Gini means are Bajraktarevi\'c means. 

Let us emphasize that Bajraktarevi\'c means are repetition invariant and have no negligible element.
Moreover, there are following properties binding these four families:
\begin{enumerate}[i.]
 \item power means are the only homogeneous quasi-arithmetic means (cf.\ \cite{HarLitPol34}, \cite{Pal00a}, \cite{Pas15a});
 \item every quasi-arithmetic mean is a Bajraktarevi\'c mean;
 \item Gini means are the only homogeneous Bajraktarevi\'c means \cite{AczDar63c};
 \item means which are simultaneously quasi-arithmetic and Gini means are exactly power means.
\end{enumerate}

The are three more families of means which will be of our interest -- Hamy means, Symmetric polynomial means and Biplanar means -- we will introduce them in section~\ref{sec:hightypes}.

\medskip 
Finally, let us mention that the family of Bajraktarevi\'c means can be generalized to so-called quasideviation means. We call a two-variable function $E \colon I \times I \to I$ to be \emph{quasi-deviation} if, 
\begin{enumerate}[(a)]
 \item $\sign(E(x,y))=\sign(x-y)$,
 \item for all $x\in I$, the map $y\mapsto E(x,y)$ is continuous and,
 \item for all $x<y$ in $I$, the mapping $(x,y)\ni t\mapsto \frac{E(y,t)}{E(x,t)}$ is strictly increasing.
\end{enumerate}

For a given quasideviation $E$ we define a \emph{quasideviation mean} $\D_E\colon \bigcup_{n=1}^\infty I^n \to I$ at a vector $x=(x_1,\dots,x_n) \in I^n$ as a unique zero of the mapping $I \ni y \mapsto \sum_{i=1}^n E(x_i,y)$ (cf.\ \cite{Pal87d}). 

In fact P\'ales \cite{Pal87d} delivered two characterizations of this family which will be presented in a subsequent propositions
\begin{prop}\label{prop:QD1}
 Let $I$ be an interval. Function $\M \colon \bigcup_{n=1}^\infty I^n \to I$ is a quasideviation mean if and only if all of the following conditions is satisfied 
 \begin{enumerate}[(i)]
  \item \label{QD1-1} $\M$ is strict, i.e. $\min(v)\le \M(v)\le \max(v)$ for all $v \in \bigcup_{n=1}^\infty I^n$ and equalities hold only for a constant vector $v$;
  \item \label{QD1-2} $\M$ is symmetric, i.e. all $n$-variable restriction $\M|_{I^n}$ a is symmetric function;
  \item \label{QD1-3} $\M$ is infinitesimal, i.e. 
  \Eq{*}{
  \lim_{k \to \infty} \max_{m \in\{1,\dots,k\}} |\M(\underbrace{x,\dots,x}_{m\text{ times}},\underbrace{y,\dots,y}_{k-m\text{ times}})-
  \M(\underbrace{x,\dots\dots,x}_{m-1\text{ times}},\underbrace{y,\dots\dots,y}_{k-m+1\text{ times}})|=0\,;
  }
  \item \label{QD1-4} for all $k \in \N$ and all vectors $\vec{x_1},\dots,\vec{x_k} \in \bigcup_{n=1}^\infty I^n$, 
  \Eq{*}{
  \min_{i \in \{1,\dots,k\}} \M(\vec{x_i}) < \M(\vec{x_1},\dots,\vec{x_k})< \max_{i \in \{1,\dots,k\}} \M(\vec{x_i})
  }
  unless all $\M(\vec{x_i})$-s are equal (coma stands for a concatenation of vectors), in the latter case this inequality becomes an equality.
 \end{enumerate}
\end{prop}

\begin{prop}\label{prop:QD2}
 Let $I$ be an interval. Function $\M \colon \bigcup_{n=1}^\infty I^n \to I$ is a quasideviation mean if and only if all of the following conditions is satisfied 
 \begin{enumerate}[(i)]
  \item \label{QD2-1} $\M$ is reflexive, i.e. $\M(\underbrace{x,\dots,x}_{n\text{ times}})=x$ for all $x \in I$ and $n \in \N$;
  \item \label{QD2-2} $\M$ is symmetric;
  \item \label{QD2-3} for all $x,y,u,v \in I$ with $x<u<v<y$ there exist $n,m \in \N$ such that 
  \Eq{*}{
  u < \M(\underbrace{x,\dots,x}_{n\text{ times}},\underbrace{y,\dots,y}_{m\text{ times}})<v;
  }
  \item \label{QD2-4} for every vectors $\vec{x},\vec{y} \in \bigcup_{n=1}^\infty I^n$ with $\M(\vec{x}) < \M(\vec{y})$ we have $\M(\vec{x})<\M(\vec{x},\vec{y})<\M(\vec{y})$;
    \item \label{QD2-5} for every vectors $\vec{x},\vec{y} \in \bigcup_{n=1}^\infty I^n$ with $\M(\vec{x}) = \M(\vec{y})$ we have $\M(\vec{x},\vec{y})=\M(\vec{y})$.
 \end{enumerate}
\end{prop}

In fact a Bajraktarevi\'c mean have a nice characterization in terms of quasideviation means. Namely Bajraktarevi\'c mean are exactly these quasideviation means which satisfies so-called linking condition (see\ \cite{Pal87d}), that is for all $\vec{x},\vec{y},\vec{u},\vec{v} \in \bigcup_{n=1}^\infty I^n$ we have
\Eq{*}{
\M(\vec{x},\vec{u}) \le \M(\vec{x},\vec{v}) \wedge \M(\vec{y},\vec{u}) \le \M(\vec{y},\vec{v}) \Longrightarrow \M(\vec{x},\vec{u},\vec{y},\vec{u}) \le \M(\vec{x},\vec{v},\vec{y},\vec{v})\:. 
}
Let us notice that in view of Propositions~\ref{prop:QD1} and \ref{prop:QD2} one can characterize Bajraktarevi\'c mean in two ways using either five or six axioms. 
\subsection{Theory of languages}
Now we introduce so-called \emph{regular languages}. This section is to provide necessary background for our consideration, but we will not refer directly to results contained here. Therefore all notions introduced here are valid till the end of this section, as it is handy. This very elementary introduction is based on Boja\'nczyk \cite{Boj12} and Hopcroft-Motwani-Ullman \cite{HopMotUll06}.

Let $\Sigma$ be finite set called \emph{alphabet}. Let $\Sigma^+$ be a set of all nonempty strings having symbols in $\Sigma$. Let $\varepsilon$ be an empty word and $\Sigma^*:=\Sigma^+ \cup \{\varepsilon\}$. Every subset $L \subseteq \Sigma^*$ is called a \emph{language}.

\emph{Deterministic finite-state automata} consists of 
\begin{enumerate}
 \item A finite set of \emph{states} denoted by $Q$;
 \item A finite set of \emph{input symbols} denoted by $\Sigma$;
 \item A \emph{transition function} $\delta \colon Q \times \Sigma \to Q$;
 \item A \emph{start state} $q_0 \in Q$;
 \item A set of \emph{accepting states} $F\subset Q$.
\end{enumerate}
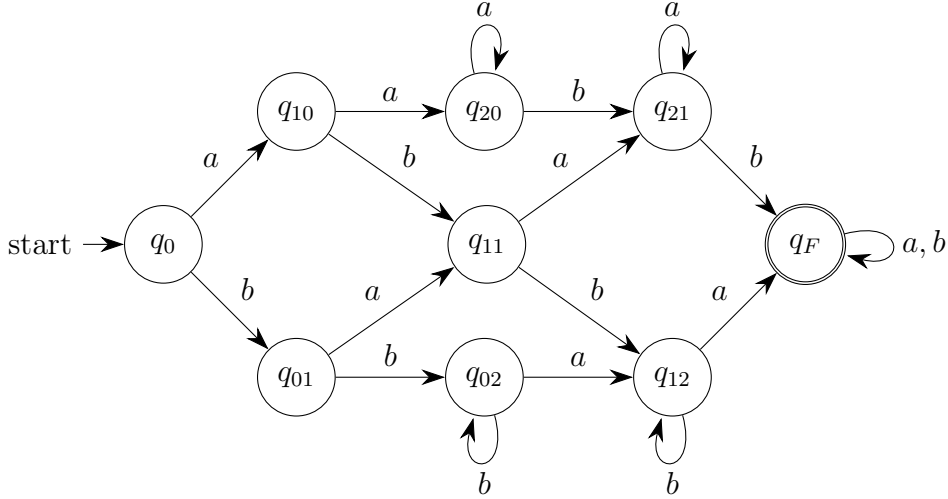
\begin{figure}

\begin{tikzpicture}[>={Stealth[width=6pt,length=9pt]},
node distance=25mm,auto]
\node[state,initial] (q_0) {$q_0$};
\node[state] (q_a) [above right of=q_0] {$q_{10}$};
\node[state] (q_b) [below right of=q_0] {$q_{01}$};
\node[state] (q_ab) at (4.3,0) {$q_{11}$};
\node[state] (q_a2) [right of= q_a] {$q_{20}$};
\node[state] (q_b2) [right of= q_b] {$q_{02}$};
\node[state] (q_a2b) [right of= q_a2] {$q_{21}$};
\node[state] (q_ab2) [right of= q_b2] {$q_{12}$};
\node[state,accepting](q_F) [below right of=q_a2b] {$q_{F}$};
\path[->] (q_0) edge node {$a$} (q_a);
\path[->] (q_0) edge node {$b$} (q_b);
\path[->] (q_b) edge node {$a$} (q_ab);
\path[->] (q_b) edge node {$b$} (q_b2);
\path[->] (q_a) edge node {$a$} (q_a2);
\path[->] (q_a) edge node {$b$} (q_ab);
\path[->] (q_a2) edge [loop above] node {$a$} (q_a2);
\path[->] (q_b2) edge [loop below] node {$b$} (q_a2);
\path[->] (q_ab) edge node {$a$} (q_a2b);
\path[->] (q_ab) edge node {$b$} (q_ab2);
\path[->] (q_a2) edge node {$b$} (q_a2b);
\path[->] (q_b2) edge node {$a$} (q_ab2);
\path[->] (q_a2b) edge [loop above] node {$a$} (q_a2b);
\path[->] (q_ab2) edge [loop below] node {$b$} (q_ab2);
\path[->] (q_a2b) edge node {$b$} (q_F);
\path[->] (q_ab2) edge node {$a$} (q_F);
\path[->] (q_F) edge [loop right] node {$a,b$} (q_F);
\end{tikzpicture}

  \caption[TBD]{\tabular[t]{@{}l@{}}\label{fig:lang}Example of deterministic finite-state automata. \\ Input symbols $\Sigma=\{a,b\}$, the only accepting state is $q_F$.\\ Corresponding language is ``at least two $a$-s and at least two $b$-s''.\endtabular}
\end{figure}

Processing of word $w=(w_1,w_2,\dots,w_n)\in \Sigma^*$ is based on iterative applying the transition function. More precisely we define a function $q \colon \Sigma^* \to Q$ by
\Eq{E:defq}{
q(\varepsilon)&:=q_0, \\
q(w_1)&:=\delta(q_0,w_1),\\
q(w_1\dots w_k)&:=\delta(q(w_1\dots w_{k-1}),w_k) &\text{ for }k >2.
}
Define $L(A):=q^{-1}(F)=\{w \in \Sigma^* \colon q(w)\in F\}$. A language is called \emph{regular} if it equals $L(A)$ for some deterministic finite-state automata $A=(Q,\Sigma,\delta,q_0,F)$.

Now we present two more approaches to regular languages. First, a language $L$ is regular if and only if there exists a finite monoid $(M,\cdot,1)$, a function $e \colon \Sigma \to M$, and a subset $F \subset M$ such that 
\Eq{*}{
(w_1,\dots,w_n)\in L \iff e(w_1)\cdots e(w_n) \in F
}
(empty word belongs to the language if an only if $1 \in F$).

Third definition is much more abstract. Define a relation $\sim$ on $\Sigma^*$ (so-called Myhill relation) by
\Eq{*}{
w \sim v :\iff \forall_{p,q \in \Sigma^*} \big(pwq \in L \iff pvq \in L \big)\:.
}

Obviously $\sim$ is an equivalence relation. Moreover it is known that $L$ is regular if and only if $\Sigma^*/_\sim$ is finite. This statement remains valid if we replace $\sim$ by one-sided Myhill relations, i.e.

\Eq{*}{
w \sim_- v :\iff \forall_{p \in \Sigma^*} \big(pw \in L \iff pv \in L \big)\:;\\
w \sim_+ v :\iff \forall_{q \in \Sigma^*} \big(wq \in L \iff vq \in L \big)\:.
}

Note that if $L$ is permutation-invariant (that is $(v_1,\dots,v_n)\in L$ if and only if $(v_{\sigma(1)}\dots,v_{\sigma
(n)})\in L$ for every permutation $\sigma \in S_n$) then all these relation are equal to each other. Example of an automata recognizing such a language is presented on Figure~\ref{fig:lang}. Note that every state of this automata refers to some element of the quotient set $\{a,b\}^* /_\sim$.

\subsection{Online evaluation}

In this section we intend to show the intuition beyond our idea. We keep the convention that all notations are internal within this section.  Following the idea of the previous section define a tuple consisting of 

\begin{enumerate}
 \item A set of \emph{states} denoted by $Q$;
 \item An interval $I$;
 \item A \emph{transition function} $\delta \colon Q \times I \to Q$;
 \item A \emph{start state} $q_0 \in Q$;
 \item An \emph{evaluation function} $F \colon Q \to I$.
\end{enumerate}

Processing of vector is based on iterative applying the transition function to obtain a 
function $q \colon \bigcup_{n=1}^\infty I^n \to Q$ defined by \eq{E:defq}. We calculate the final value using the evaluation function, i.e. the outcome of our calculations is $F\circ q \colon \bigcup_{n=1}^\infty I^n \to I$ (therefore its domain and set of values coincide with the one which is characteristic for a mean).

There appear a natural question: why is this consideration so important? Assume that we are given a sequence of elements in $I$. There are essentially two ways of input such sequence (to the computer). The first one is to allocate memory to store all sequence (\emph{offline input}). Main difficulty is that we have to know in advance (at least an upper bound to) a number 
of elements. (In practice we can also allocate memory during the input but it has no reasonable interpretation in ZFC theory.)

The second way of processing is so-called \emph{online input}. In this method we have some special terminating symbol (denoted here by $\#$) which appears at the end of input. The algorithm process a vector as follows:

\begin{algorithmic}[0] 
\Procedure{OnlineEvaluation}{$Q,I,\delta,q_0,F$}
\State $q\gets q_0$ \Comment{attach the initial state}
\While{True} \Comment{repeat forever}
\State \textbf{input} $a \in I \cup \{\#\}$
\If{$a = \#$} \Comment{terminating symbol was given}
\State \textbf{return} $F(q)$ \Comment{evaluate the function at the end of the sequence}
\Else
\State $q \gets \delta(q,a)$ \Comment{single transition}
\EndIf
\EndWhile
\EndProcedure
\end{algorithmic}

In this method data is a stream, i.e. after each element we can either add next one or evaluate the final value. Note that we need to keep in memory only a state $q \in Q$ (similarly like in a case of automata).

We will follow this idea (with simplified structure) to evaluate premeans. This is the reason why we refer to them as \emph{online premeans}.

\section{Online premeans}
Let $I \subset \R$ be an interval, $(Y,+)$ be a commutative semigroup, $F \colon  I \to Y$, and $G \colon \omega F(I) \to I$ such that $G (nF(x))=x$ for all $n \in \N$ and $x \in I$, 
where $\omega F(I)$ is a union of all Minkowski sums, i.e.
\Eq{*}{
\omega F(I):=\bigcup_{n=1}^\infty  nF(I)=\bigcup_{n=1}^\infty  \underbrace{F(I)+F(I)+\dots+F(I)}_{n\text{ times}}.
}
Define \emph{online premean} $\Lv_{F,G} \colon \bigcup_{n=1}^\infty I^n \to I$ by
\Eq{Lvdef}{
\Lv_{F,G}(a_1,\dots,a_n):= G \Big( F(a_1)+\cdots+F(a_n) \Big).
}
A pair $(F,G)$ is called a \emph{generating pair of $\Lv_{F,G}$}, $Y$ is called a \emph{freedom space}.
It is easy to verify that online premeans are reflexive and symmetric.

The intuition beyond this name is quite natural. Having a value $F(a_1)+\cdots+F(a_k) \in Y$ we can decide whether we would like to continue adding elements (i.e. add $F(a_{k+1})$ and so on) of to finish (that is to apply function $G$ to this value), exactly like in {\sc OnlineEvaluation} procedure. Let us also stress the analogy between this setting and the monoid approach to regular languages.

Before we begin dealing with this family let us present a simple example explaining our idea.
\begin{exa}[Gini means] \label{ex:Gini}
For $p,q \in \R$ define $F\colon \R_+ \to D$ and $G \colon D \to \R_+$ 
($D=\R_+^2$ for $p \ne q$ and $D=\R_+ \times \R$ for $p=q$) by
\Eq{*}{
F(x):=\begin{cases} (x^p,x^q) &\text{ for } p \ne q, \\
       (x^p\ln x,x^p) &\text{ for } p = q,
      \end{cases}
\qquad 
G(x,y):=\begin{cases} \big(\frac xy\big)^{\frac1{p-q}} &\text{ for } p \ne q, \\
       \exp\big(\frac xy\big) &\text{ for } p = q.
      \end{cases}
}
Then it is easy to verify that $\Lv_{F,G}=\G_{p,q}$ for all $p,q \in \R$. 
 \end{exa}

Now let us present few preliminary observations concerning online premeans.  

\begin{rem} \label{rem:1-1onto}
By $G\circ F = \id_I$ we obtain that $F$ is 1-1 and $G$ is onto.
\end{rem}

\begin{rem} Note that if $\Lv_{F,G}$ defined by \eq{Lvdef} is a premean on $I$ then we have $G(nF(x)) =x$ for all $x \in I$ and $n \in \N$. Therefore we do not have to verify this condition provided \eq{Lvdef} defines a premean on $I$.
\end{rem}

\begin{rem}[P\'ales \cite{Pal89b}] \label{rem:premean}
 Let $I$ be an interval $(Y,\oplus)$ be a free abelian semigroup generated by the elements of $I$. Then there is a natural 1-1 correspondence between symmetric functions $\M \colon \bigcup_{n=1}^\infty I^n \to \R$ and functions $m \colon Y \to I$ given by the formula
 \Eq{*}{
m(x_1 \oplus \cdots \oplus x_n):=\M(x_1,\dots,x_n) \qquad (n \in \N,\, x \in I^n).
 }
 Then $\Lv_{\id,m}=\M$, where $\id$ is a natural embedding $I \hookrightarrow Y$.
%
 In particular every symmetric mean is an online premean. In fact a converse implication is also valid as every online premean is symmetric.
\end{rem} 

\begin{rem}\label{rem:myhill} Similarly like in a case of languages, for a given symmetric mean $\M$ we can define a Myhill-type relation $\sim$ on $\bigcup_{n=1}^\infty I^n$ by
\Eq{*}{
v \sim w :\iff \Big( \M(v)=\M(w) \text{ and } \M(v,q)=\M(w,q) \text{ for all }q \in \bigcup_{n=1}^\infty I^n \Big).
}
Obviously $\sim$ is an equivalence relation. Moreover, $Y:=\big(\bigcup_{n=1}^\infty I^n\big) /_{\sim}$ is the minimal freedom space for a mean $\M$.
\end{rem}

Due to Remark~\ref{rem:myhill} freedom space describe amount of information which should be preserved from a sequence in order to evaluate a mean and/or add elements. The aim now is to minimalize freedom space for a given premean in a constructive way (abstract construction has been described in a remark above). 

If freedom space has some topology then the following simple proposition is very useful:
\begin{prop}\label{prop:contem}
Let $Y$ a be topological semigroup, $Y'$ be its subspace (which is also a semigroup), and $\imath \colon Y' \hookrightarrow Y$ be the inclusion map. Moreover let $F \colon I \to Y'$ and $G \colon \omega F(Y') \to I$ be two continuous functions. 

Then
both $F^*:=\imath \circ F \colon I \to Y$ and $G^*:=G \circ \imath^{-1}|_{\omega (i\circ F(I))} \colon \omega (i\circ F(I)) \to I$ are continuous. Moreover $\Lv_{F,G}=\Lv_{F^*,G^*}$.
\end{prop}

The intuition beyond this proposition is very natural. Namely, if a freedom space is larger (one is embedable into another) then we it can preserve more information -- finite anolgue is a number of equivalnce classes of Myhill relation, however it this setup it is always continuum.

Consequently we may compare complexity of means based of their freedom spaces (or minimal freedom spaces). Indeed, in view of Proposition~\ref{prop:contem} one can say that a mean $\M_1$ is \emph{computationally simpler} that a mean $\M_2$ if a freedom space of $\M_1$ is continuously embeddable into a freedom space of $\M_2$. However, as it was mentioned, freedom space is difficult to calculate a'priori, furthermore there are no natural topology for a given freedom space, finally it this condition would be very difficult to satisfy or even verify.

Therefore instead of taking minimal freedom spaces we will consider a sort of reference order. Namely we assume that either $Y=\R^k$ or $Y=\R^k \times \Z$ ($k \in \N$) and assume that both $F$ and $G$ are continuous.

\subsection{Types of premeans}

A mean $\M$ is of type $\TYPE k$ ($k\in \N$) if $\M=\Lv_{F,G}$ for some continuous functions $F \colon I \to (\R^k,+)$ and $G \colon \omega F(I) \to I$.

Analogously, if it is true with $F \colon I \to (\R^k \times \Z,+)$ then we say that $\M$ is of type $\TYPEP k$ ($k \in \N$). In this case as $F$ is continuous, it is constant on the last entry. Thus we can assume without loss of generality that it equals one on this coordinate (it is also the reason why there is no point to put more than one integer entry). Base of this we refer to this element as a \emph{counter}.

In fact we slightly abuse these notions and denote by $\TYPE{k}$ and $\TYPEP{k}$ classes of all means of this type (defined on any interval).

As for all $k \in \N$ there exist continuous embeddings $(\R^k,+) \hookrightarrow (\R^k \times \Z,+) \hookrightarrow (\R^{k+1},+)$, in view of Proposition~\ref{prop:contem} we obtain a series of inclusions 
\Eq{E:Hier}{
\TYPE{1} \subseteq \TYPEP{1} \subseteq \TYPE{2} \subseteq \TYPEP{2} \subseteq \dots
}

%

The remaining part of paper goes twofold. First, we characterize means of types $\TYPE{1}$, $\TYPEP{1}$, and repetition invariant $\TYPE{2}$ means without negligable elements -- we obtain empty set, quasi-arithmetic means, and Bajraktarevi\'c means, respectively. 
Later, applying generalized symmetric polynomials, we show some examples of means of higher type. 

To conclude this section let us stress that there are means which are not of any of these types. For example median (lower- or upper-) is one of them. For instance two vectors are in (Myhill-type) relation for a median if an only if one is a permutation of another (proof of this statement is straightforward and therefore omitted). 

\section{Means of low types}
\subsection{Means of type $\TYPE{1}$ and $\TYPEP{1}$}
We begin our consideration with two first classes in a hierarchy mentioned in \eq{E:Hier}. These are the only two classes which reduces to well-known families (empty set and quasi-arithmetic means, respectively). These result are proved in subsequent propositions.

\begin{prop}\label{prop:T1}
 There are no means of type $\TYPE{1}$.
\end{prop}
\begin{proof}
 Assume that $\M$ is a mean of type $\TYPE{1}$. Then there exists continuous functions $F \colon I \to \R$ and $G \colon \omega F(I) \to I$ such that $\M$ is of the form \eq{Lvdef}. As $F$ is continuous we obtain that it is strictly monotone and therefore invertible. 
 Thus there exist $x_1,x_2 \in I$ with $x_1\ne x_2$ such that $F(x_1)$ and $F(x_2)$ are both nonzero, have the same sign, and $\tfrac{F(x_1)}{F(x_2)}$ is a rational number, i.e. $\tfrac{F(x_1)}{F(x_2)}=\tfrac pq$ for some $p,q \in \N_+$.
Then $qF(x_1)=pF(x_2)$ and, consequently,
\Eq{*}{
x_1
=\Lv_{F,G}(\underbrace{x_1,\dots,x_1}_{q \text{ times}})
=G(qF(x_1))
=G(pF(x_2))
=\Lv_{F,G}(\underbrace{x_2,\dots,x_2}_{p \text{ times}})
=x_2
}
contradicting the assumption.
\end{proof}

\begin{prop}\label{prop:T1+}
Means of type $\TYPEP{1}$ are exactly quasi-arithmetic means.
\end{prop}
\begin{proof}
 To verify that every quasi-arithmetic mean is of type $\TYPEP{1}$ take any interval $I$ and continuous, strictly monotone function $f \colon I \to \R$. Define $F(x)=(f(x),1)$ and $G(y,n):=f^{-1}(y/n)$. For $n \in \N$ and $(v_1,\dots,v_n)\in I^n$ we get
\Eq{*}{
\Lv_{F,G}(v_1,\dots,v_n)
&=G\Big(\sum_{i=1}^n F(v_i)\Big)
=G\Big(\sum_{i=1}^n (f(v_i),1)\Big)
=G\Big(\sum_{i=1}^nf(v_i),n\Big)\\
&=f^{-1}\Big(\tfrac1n\sum_{i=1}^nf(v_i)\Big)
=\QA{f}(v_1,\dots,v_n).
}
Thus $\Lv_{F,G}=\QA{f}$ which implies that $\QA{f}$ is of type $\TYPEP{1}$.
  
  We are now going to prove that every mean of type $\TYPEP{1}$ is a quasi-arithmetic mean. Take a single variable, continuous function $f \colon I \to R$ such that $F(x)=(f(x),1)$. In view of Remark~
  \ref{rem:1-1onto} we obtain that $f$ is 1-1 as it is also continuous we get that it is strictly monotone. 
    
In view of reflexivity of $\Lv_{F,G}$ we have 
\Eq{*}{
x
=\Lv_{F,G}(x,\dots,x)=G(nF(x))=G(nf(x),n) \qquad (x \in I,\, n \in \N).
}
In particular, upon putting $x:=f^{-1}(y/n)$ for $n \in \N$ and $y \in \{ny \colon y \in f(I)\}$, we obtain
\Eq{*}{
G(y,n)=f^{-1}(y/n) \qquad (y \in \{ny \colon y \in f(I)\},\, n \in \N).
}

As $f$ is continuous and strictly monotone we have $\{ny \colon y \in f(I)\}=n f(I)$. 
Repeating the same argumentation as in the previous implication we obtain $\Lv_{F,G}=\QA{f}$.
\end{proof}

\subsection{Means of type $\TYPE{2}$}
This section consists of a single statement which characterize Bajraktarevi\'c means. Next in two examples we show that its assumptions cannot be omitted. 

In fact this statement was our motivation to write this paper. It turns out that we can characterize Bajraktarevi\'c mean using their complexity and very natural axioms. Its long and technical proof is shifted to the last section.

\begin{thm}\label{thm:RepInvT2}
Repetition invariant means of type $\TYPE{2}$ without negligible element are exactly Bajraktarevi\'c means.
\end{thm}

In fact this theorem is somehow related with Propositions \ref{prop:QD1} and \ref{prop:QD2} as in view of this theorem the following statements are easy to verify
\begin{prop} The following statements remains valid:
\begin{enumerate}[A.]
\item Each of conditions:  \eq{QD1-1}, \eq{QD1-3}, \eq{QD1-4} in Proposition~\ref{prop:QD1}, and
\eq{QD2-3}, \eq{QD2-4} in Proposition~\ref{prop:QD2} 
implies that a mean has no negligable element.
\item Each of conditions: 
\eq{QD1-4} in Proposition~\ref{prop:QD1}, and
\eq{QD2-5} in Proposition~\ref{prop:QD2} implies that a mean is repetition invariant.
\end{enumerate}
\end{prop}

In fact this proposition provide us (in total) five choices of axioms to guarantee that both assumptions in Theorem~\ref{thm:RepInvT2} are simultaneously valid (obviously none of them implies directly that a mean is $\TYPE{2}$).

In the following two examples we show that repetition invariance and having no negligible element in the theorem above is unavoidable, respectively.
\begin{exa}
 Let $F \colon [3,4] \to \R^2$ and $G \colon \R_+^2 \to \R$  be given by $F(x):=(x^2,x)$ and $G(r,s)=r/s$, respectively. 
 
 Then we obtain, for all $n \in \N$ and $a \in [3,4]^n$,
\Eq{*}{
\Lv_{F,G}(a_1,\dots,a_n)=G(\sum_{i=1}^n F(a_i))=G(\sum_{i=1}^na_i^2,\sum_{i=1}^na_i)=\frac{\sum_{i=1}^na_i^2}{\sum_{i=1}^na_i}=\G_{2,1}(a_1,\dots,a_n),
}
that is $\Lv_{F,G}=\G_{2,1}|_{[3,4]}$.

Let $H \colon \R \times \big( [3,4] \cup [6,8] \cup [9,\infty) \big) \to \R$ be given by
\Eq{*}{
H(r,s):=
\begin{cases}
s &\qquad \text{ for } s \in [3,4];\\
s/2 &\qquad \text{ for } s \in [6,8];\\
r/s &\qquad \text{ for } s\in [9,+\infty). 
\end{cases}
}

\noindent In the simplest case, $n=1$, as $x_1 \in [3,4]$ we have $\Lv_{F,H}(x_1)=H(x_1^2,x_1)=x_1$.

\noindent For $n=2$ one gets $x_1+x_2 \in [6,8]$,  thus 
$\Lv_{F,H}(x_1,x_2)=H(x_1^2+x_2^2,x_1+x_2)=\tfrac12(x_1+x_2)$.

For $n\ge3$ we obtain $x_1+\dots+x_k \ge 9$, whence
\Eq{*}{
\Lv_{F,H}(x_1,\dots,x_n)=
\Lv_{F,G}(x_1,\dots,x_n)=\G_{2,1}(x_1,\dots,x_n)\:.
}

Binding all cases altogether we obtain
\Eq{*}{
\Lv_{F,H}(x_1,\dots,x_n)=
\begin{cases}
x_1 &\text{ for }n =1\,;\\
\frac{x_1+x_2}2 &\text{ for }n =2\,;\\
 \frac{x_1^2+\cdots+x_n^2}{x_1+\cdots+x_n}&\text{ for }n \ge 3\,.
\end{cases}
}
Then $\Lv_{F,H}$ is not repetition invariant as 
$\Lv_{F,H}(3,4)=\tfrac72\ne\tfrac{25}7 =\Lv_{F,H}(3,3,4,4)$.
Thus $\Lv_{F,H}$ is a mean of type $\TYPE{2}$ which is not a Bajraktarevi\'c mean. 

Obviously we also have $\Lv_{F,G}$ and $\Lv_{F,H}$ are two different means, which implies that the function $F$ is not alone sufficient to determine a mean.
\end{exa}

\begin{exa}
 Let $F \colon \R\to \R^2$ be given by $F(x)=(x^3,x^2)$. 
 Let $D:=\{(u,v)\in \R^2\colon \abs{u}\le \abs{v}^{3/2}\}$ and consider a function $G\colon D \to \R$ given by
 \Eq{*}{
 G(u,v)&:=
 \begin{cases}
 0 \qquad &\text{ for }(u,v)=(0,0);\\
u/v \qquad &\text{ otherwise}.
\end{cases}
 }
 
First we need to provide that $G$ is continuous. In fact the only nontrivial point is $(0,0)$. However, for $\varepsilon>0$, $v \in (-\varepsilon,\varepsilon) \setminus\{0\}$ and $u \in (-|v|^{3/2},|v|^{3/2})$ we have 
\Eq{*}{
\abs{G(u,v)}=\abs{u\cdot \frac{u^2}{v^3}}^{1/3}\le|u|^{1/3}\le|v|^{1/2}\le\varepsilon^{1/2},
}
what implies that $G$ is a continuous at $(0,0)$ and, as a consequence, $G$ is continuous.

Now we prove that $\omega F(\R) \subset D$ or, equivalently,
\Eq{G23CE}{
\Big|\sum_{i=1}^n x_i^3\Big| \le \Big|\sum_{i=1}^n x_i^2\Big|^{3/2} \text{ for all }n \in \N \text{ and }x \in \R^n.
}
One can assume that all $x_i$-s are positive (by omitting all zeros and replacing $x_i$ by $|x_i|$) and rewrite \eq{G23CE} in a form
\Eq{*}{
\G_{3,2}(x_1,\dots,x_n)^2 \le \sum_{i=1}^n x_n^2 \qquad (n\in \N,\,x\in\R_+^n).
}

However $\G_{3,2}(x_1,\dots,x_n)^2 \le \max(x_1,\dots,x_n)^2\le x_1^2+\dots+x_n^2$, what ends the proof of \eq{G23CE}. Thus we obtain
\Eq{*}{
\Lv_{F,G}(x_1,\dots,x_n)=
\begin{cases}
0 & \text{for }x_1=\cdots=x_n=0,\\
\frac{x_1^3+\dots+x_n^3}{x_1^2+\dots+x_n^2} & \text{otherwise.}
\end{cases}
}
This is obviously the continuous, repetition invariant mean defined of $\R$ with a negligible element (equals $0$) which is of type $\TYPE{2}$. Thus $\Lv_{F,G}$ is not a Bajraktarevi\'c mean as these means have no negligible element. In the same way $\Lv_{F,G}$ is not a deviation (or even a semideviation) mean.
\end{exa}

\section{\label{sec:hightypes}Means of higher types}

At the moment we intend to show some means of higher types. To this end, for every vector $(x_1,\dots,x_n)$ having all positive entries define 
\Eq{*}{
\gamma_{p_1,\dots,p_s}&:=
\begin{cases}\sum\limits_{\substack{i_1,\dots,i_s\\ i_k \ne i_l}}x_{i_1}^{p_1} \cdots x_{i_s}^{p_s} &\text { for }s \ge n;\\[1mm]
0 & \text{ for }s <n.
\end{cases}\\
\sigma_{s,p}&:=\dfrac{\gamma_{p,\dots,p}}{s!}=\begin{cases}\sum\limits_{1\le i_1<\cdots<i_s \le n}x_{i_1}^{p} \cdots x_{i_s}^{p} &\text { for }s \ge n;\\[1mm]
0 & \text{ for }s <n.
\end{cases}
}
The following technical lemma is of essential importance 
\begin{lem} \label{lem:sym}
Let $s \in \N$ and $(p_1,\dots,p_s)$ be a vector of real numbers. Define the set 
\Eq{*}{
Q:=\Big\{\sum_{i \in T} p_i \colon T \subset \{1,\dots,s\} \wedge T \ne \emptyset\Big\}.
}
Then $\gamma_{p_1,\dots,p_s} \in \Z[\gamma_q \colon q \in Q]$. In particular $\sigma_{s,p} \in \Z[\gamma_p,\gamma_{2p},\dots,\gamma_{sp}]$.
\end{lem}

\begin{proof}
 We prove it by induction with respect to $s$. 
For $s=1$ we get $\gamma_{p_1} \in \Z[\gamma_{p_1}]$ which is a trivial statement.

Now take and vector $(p_0,p_1,\dots,p_s)$ of real numbers.
Then we can easy verify that whenever $n> s$ we have
\Eq{*}{
\gamma_{p_1,\dots,p_s}\gamma_{p_0}&=\gamma_{p_0,p_1,\dots,p_s}
+\gamma_{p_1+p_0,p_2,\dots,p_s}
+\gamma_{p_1,p_2+p_0,\dots,p_s}
+\dots
+\gamma_{p_1,p_2,\dots,p_s+p_0},\\
\gamma_{p_0,p_1,\dots,p_s}&=
\gamma_{p_1,\dots,p_s}\gamma_{p_0}
-\gamma_{p_1+p_0,p_2,\dots,p_s}
-\gamma_{p_1,p_2+p_0,\dots,p_s}
-\dots
-\gamma_{p_1,p_2,\dots,p_s+p_0}.
}
Therefore
\Eq{Alg1}{
\gamma_{p_0,\dots,p_s} \in \Z[\gamma_{p_0},\gamma_{p_1,\dots,p_s}
,\gamma_{p_1+p_0,p_2,\dots,p_s}
,\gamma_{p_1,p_2+p_0,\dots,p_s}
,\dots
,\gamma_{p_1,p_2,\dots,p_s+p_0}
]\:.
}
However, if one define
\Eq{*}{
Q_0:= \Big\{\sum_{i \in T} p_i \colon T \subset \{0,\dots,s\} \wedge T \ne \emptyset\Big\}
}

Then $\gamma_{p_1,\dots,p_s} \in \Z[\gamma_q \colon q \in Q] \subset \Z[\gamma_q \colon q \in Q_0]$. Furthermore, using inductive assumption we obtain 
\Eq{*}{
\gamma_{p_1,p_2,\dots,p_i+p_0,\dots,p_s} \in \Z[\gamma_q \colon q \in Q_0] \qquad \text{ for all }i \in \{1,\dots,s\}.
}
Finally, in view of \eq{Alg1} one gets $\gamma_{p_0,\dots,p_s} \in \Z[\gamma_q \colon q \in Q_0]$ what concludes the proof.
\end{proof}

\begin{exa}[Hamy means]
Let $r \in \N$ and $\mathfrak{ha}_r \colon \bigcup_{n=1}^\infty \R_+^n \to \R_+$ be given by
\Eq{*}{
\mathfrak{ha}_r(x_1,\dots,x_n):=
\begin{cases}
 \frac{x_1+\dots+x_n}{n} &\text{ for }n < r \\
 {\binom nr}^{-1} \sum\limits_{1\le i_1<\cdots<i_r \le n} \sqrt[r]{x_{i_1}\cdots x_{i_r}} & \text{ for }n \ge r
\end{cases}
}

One can rewrite it in a compact form 
\Eq{*}{
\mathfrak{ha}_r(x_1,\dots,x_n):= 
\begin{cases}
n^{-1}\gamma_1 & \text{ for }n < r, \\
{\binom nr}^{-1}\sigma_{r,1/r} & \text{ for }n \ge r.
\end{cases}
}
Using Lemma~\ref{lem:sym} we have $\sigma_{r,1/r}\in \Z[\gamma_{1/r},\gamma_{2/r},\dots,\gamma_{r/r}]$. Based on this there exists a continuous function $G \colon \R^{r+1} \to \R$ such that $\mathfrak{ha}_r(x_1,\dots,x_n)=G(\gamma_{1/r},\gamma_{2/r},\dots,\gamma_{r/r},n)$ what implies that $\mathfrak{ha}_r \in \TYPEP{r}$. 

Let us emphasize that it is still an open problem if $\mathfrak{ha}_r \in \TYPE{r}$. In a sense we obtain only some upper estimation of the complexity of Hamy means.
\end{exa}

\begin{exa}[Symmetric polynomial means]
Let $r \in \N$ and $\mathfrak{s}_r \colon \bigcup_{n=1}^\infty \R_+^n \to \R_+$ be given by
\Eq{*}{
\mathfrak{s}_r(x_1,\dots,x_n):=
\begin{cases}
 \frac{x_1+\dots+x_n}{n} &\text{ for }n < r \\
\Big({\binom nr}^{-1} 
\sum\limits_{1\le i_1<\cdots<i_r \le n} {x_{i_1}\cdots x_{i_r}}\Big)^{1/r}
 & \text{ for }n \ge r
\end{cases}
} 
Similarly like in a case of Hamy means we obtain $\mathfrak{s}_r \in \TYPEP{r}$ for $r \in \N$.
\end{exa}

Let us emphasize that, for all $r \in \N$ and $n \ge r$ we have
$\big(\mathfrak{ha}_r(x_1^r,\dots,x_n^r)\big)^{1/r}=\mathfrak{s}_r(x_1,\dots,x_n)$.
This property refers to so-called conjugation of means; see \cite{ChuPalPas19} and \cite[section~VI.4.2]{Bul03}. In fact we can prove that, for each $r\in \N$, the means $\mathfrak{ha}_r$ and $\mathfrak{s}_r$ have the same complexity and, moreover, their minimal freedom spaces are isomporphic.

\begin{exa}[Biplanar means]
 For $p,q \in \R$ and $c,d \in \N$ with $cp\ne dq$ define a mean on $\R_+$ by
 \Eq{*}{
 \mathfrak{Bi}_{p,q,c,d}(x_1,\dots,x_n):=\begin{cases}
\left( \dfrac{{n \choose d}\sigma_{c,p}}{{n \choose c}\sigma_{d,q}}\right)^{1/(cp-dq)} &\qquad \text{whenever }n \ge \max(c,d),\\
\P_p(x_1,\dots,x_n) &\qquad\text{otherwise.}
                                         \end{cases}
 }

Then $ \mathfrak{Bi}_{p,q,c,d}$ is a function of $(\gamma_0,\gamma_p,\gamma_{2p},\dots,\gamma_{cp},\gamma_q,\gamma_{2q},\dots,\gamma_{dq})$. Thus $\mathfrak{Bi}_{p,q,c,d}  \in\TYPEP{k}$, where $k:=|\{p,2p,\dots,cp,q,2q,\dots,dq\}|$. 

We have a trivial inequality $k \le c+d$ which lead to a fact that $\mathfrak{Bi}_{p,q,c,d} \in \TYPEP{c+d}$, however we can obtain better estimations in a particular cases. 
Indeed, for $(p,q,c,d):=(2,3,3,3)$ we have $k=|\{2,4,6,3,6,9\}|=5$ so  $\mathfrak{Bi}_{2,3,3,3}\in \TYPEP5$ (instead of $\TYPEP{3+3}=\TYPEP6$).
\end{exa}

\section{Proof of Theorem~\ref{thm:RepInvT2} and auxiliary results}

First, it is easy to verify that Bajraktarevi\'c means have no negligible elements. Indeed, assume that $e$ is a negligible element. Then we have, for all $x \in I$,
\Eq{*}{
x=\B_{f,g}(x,e)&=\Big(\frac{f}{g}\Big)^{-1}\bigg(\frac{f(x)+f(e)}
                      {g(x)+g(e)}\bigg) \\
\frac{f(x)}{g(x)}&=\frac{f(x)+f(e)}
                      {g(x)+g(e)}\\
f(x)g(x)+f(x)g(e) &=f(x)g(x)+g(x)f(e) \\
\tfrac{f(x)}{g(x)} g(e)&=f(e)
}

\noindent As $\tfrac fg$ is a 1-1 we obtain $f(e)=g(e)=0$ contradicting the assumption. Second, we can easily check that Bajraktarevi\'c means are repetition invariant $\TYPE{2}$ means. The nontrivial part is to reverse this implication.

\begin{rem} \label{rem:inv}
Take $n\in \N$ and a pair of functions $F \colon I \to \R^n$ and $G \colon \R^n \to I$. Then for every invertible linear mapping $L \colon \R^n \to \R^n$, we have $\Lv_{F,G}=\Lv_{L\circ F,G \circ L^{-1}}$.
\end{rem}
 

\begin{lem}\label{lem:pom1}
Let $I$ be a closed interval and $F \colon I \to \R^2$ and $G \colon \R^2 \to I$ be two continuous functions such that $\Lv_{F,G}$ is a repetition invariant mean without negligable element.

If $F(x)=(f(x),g(x))$ for $f,g\colon I \to \R$ then there exists $\alpha,\beta \in \R$ such that the function $\alpha f+\beta g$ is nowhere vanishing.
\end{lem}

\begin{proof}
Note that $f + \alpha g$ is vanishing if and only if $\alpha \in (-f/g)(I)$, similarly $g+ \beta f$ is vanishing at some point if and only if $\beta \in -(g/f)(I)$.

Therefore either $\alpha f+\beta g$ is nonvanishing for some numbers $\alpha,\beta \in \R$ or 
\Eq{ERim}{
(\tfrac fg)(\{x \in I \colon g(x)\ne 0\} )=(\tfrac gf)(\{x \in I \colon f(x)\ne 0\} )=\R.
}

Indeed, if say $\beta \notin (\tfrac fg)(\{x \in I \colon g(x)\ne 0\} )$ then $\beta g(x) \ne f(x)$ for all $x \in I$ (recall that $g(x)=0$ implies $f(x) \ne 0$). Then obviously $f-\beta g$ is nonvanishing. The second equality is analogous.

But \eq{ERim} implies (as $I$ is a closed set and $F(x) \ne (0,0)$) that there exist four points $x_1,x_2,x_3,x_4 \in I$ such that:
\Eq{*}{
f(x_1)=0 \text{ and } g(x_1) >0 \qquad 
f(x_2)=0 \text{ and } g(x_2) <0 \\
f(x_3)>0 \text{ and } g(x_3) =0 \qquad
f(x_4)<0 \text{ and } g(x_4) =0 \\
}

Then $\omega F(I) \subset (\N f(x_3)+\N f(x_4)) \times (\N g(x_1)+\N g(x_2))$. Define $\delta:=2\max\{g(x_1),-g(x_2),f(x_3),-f(x_4)\}$.

By \cite{BanJabJab19}, $F(I)+F(I)$ has a nonempty interior, i.e. $R_r(x,y) \subset F(I)+F(I)$ for some $r>0$ and $(x,y) \in \R$, where $R_r(x,y):=[x-r,x+r] \times [y-r,y+r]$. Thus $R_{kr}(kx,ky) \subset 2k F(I)$. Take $k_0 \in \N$ such that $k_0r>\delta$. Then $R_{\delta}(k_0x,k_0y) \subset 2k_0F(I)$.

In view of definition of $\delta$ there exists $C_1,C_2,C_3,C_4 \in \N$ such that 
\Eq{*}{
(C_3f(x_3)+C_4f(x_4),C_1 g(x_1)+C_2g(x_2)) \in R_{\delta/2}(-k_0x,-k_0y).
}
Thus
\Eq{*}{
(2k_0+C_1+C_2+C_3+C_4) F(I) &\supset 2k_0 F(I)+ (C_3f(x_3)+C_4f(x_4),C_1 g(x_1)+C_2g(x_2)) \\
(2k_0+C_1+C_2+C_3+C_4) F(I)&\supset R_\delta(k_0x,k_0y)+\mu \quad \text{ for some } \mu \in R_{\delta/2}(-k_0x,-k_0y)\\
(2k_0+C_1+C_2+C_3+C_4) F(I)&\supset R_\delta(0,0)+\mu \quad \text{ for some } \mu \in R_{\delta/2}(0,0) \\
(2k_0+C_1+C_2+C_3+C_4) F(I)&\supset R_{\delta/2}(0,0).
}

It implies $\omega F(I) \supset \omega R_{\delta/2}(0,0)=\R^2$. As $\Lv_{F,G}$ is repetition invariant we have $G(\N_+ \cdot y)=G(y)$ for all $y \in \R^2$ and, applying this equality twice, $G(\Q_+ \cdot y)=G(y)$ for all $y \in \R^2$. As $G$ is continuous at $(0,0)$ we obtain $G$ that $G$ is constant contradicting Remark~\ref{rem:1-1onto}. 
\end{proof}
%
%
%
%

\begin{lem}\label{lem:Jabl}
 Let $X \subset \R_+^2$ be a connected set. Then either $X$ is contained in a line or for all directions $\lambda \in \interior \{\lambda \in S^1 \colon \lambda\cdot \R_+ \cap X \ne \emptyset \}=:\Lambda$ there exists $M_\lambda \in \R$ such that $M\cdot \lambda \in  \omega X$ for all $M > M_\lambda$.
\end{lem}
\begin{proof}
 By \cite{BanJabJab19} there exists a ball $R_{\varepsilon}(s) \subset X+X$ with a ratio $\varepsilon>0$ and a center $s \in \R^2$. Define a projection $\pi \colon \R_+^2 \to S^1$ and introduce a natural order $\prec$ on $S^1 \cap \R_+^2$. Then, by the definition, $\Lambda = \interior \pi(X)$.
 
 Fix $\lambda \in \Lambda$ and $\xi_0 \in \pi^{-1}(\lambda)$. 
 
There exist elements $\xi_-,\xi_+ \in R_{\varepsilon/4}(\xi_0) \cap X$ such that $\pi(\xi_-) \prec \pi(\xi_0)\prec \pi(\xi_+)$. Then, for some $K>0$, the set $G:=\N \xi_-+\N \xi_+$ is an $\tfrac\varepsilon2$-net in a set 
\Eq{*}{
A:=\big\{a \in \R_+^2 \colon \norma{a}_2\ge K \text{ and }\pi(\xi_-)\prec \pi(a) \prec \pi(\xi+) \big\}.
}
 Consequently $\omega X \supset G+R_\varepsilon(s) \supset s+A$. 
 
To prove that there exists $M_\lambda\in \N$ such that $M \lambda \in \omega X$ for all $M>M_\lambda$. As a metter of fact we prove the same with $\omega X$ replaced by $s+A$. Equivalently $ \lambda-\tfrac{1}{M} s \in \tfrac1M \cdot A$ for all $M>M_\lambda$. 

As for $M'>M$ we have $\tfrac1M \cdot A \subset \tfrac1{M'} \cdot A$, it suffices to prove that there exists a pair $(K_\lambda,M_\lambda)$ such that $K_\lambda>M_\lambda$ and 
$ \lambda -\tfrac{1}M s \in \tfrac1{K_\lambda} \cdot A
=[\tfrac{K}{K_\lambda},\infty) \cdot [\pi(\xi_-),\pi(\xi+)]$ for all $M>M_\lambda$. 

Continuing, it suffices to prove that for $K_\lambda>M_\lambda$ we have
\Eq{E:burza}{
R_{s/M_\lambda}(\lambda) \subset [\tfrac{K}{K_\lambda},\infty) \cdot [\pi(\xi_-),\pi(\xi+)].
}
To conclude the proof take: first $M_\lambda>0$ such that $R_{s/M_\lambda}(\lambda)
\subset (0,+\infty) \cdot [\pi(\xi_-),\pi(\xi+)]$. Second, $K_\lambda>M_\lambda$ such that \eq{E:burza} holds.
\end{proof}

\begin{lem}\label{lem:pom2}
 Under the assumption of lemma~\ref{lem:pom1} suppose additionally that $g$ is positive on its domain. Then the ratio $f/g$ is an injective function.
\end{lem}

%
%
%

\begin{proof}
Assume that there exists $\alpha \in \R$ such that the set $A:= \{x \in I \colon f(x)/g(x)=\alpha\}$ contains more than one element. 

If $A$ contains some interval $V$ then $f(x)=\alpha g(x)$ for all $x \in V$. In particular $\Lv_{F,\cdot}$ restricted to $V$ equals $\Lv_{f,\cdot}$, i.e. it is of type $\TYPE{1}$ which lead to a contradiction as by Proposition~\ref{prop:T1} there are no means of this type. 

As $A$ is a close subset of $I$, the only remaining case is that $A$ has a gap. More precisely there exists $p,q \in A$ such that $p<q$ and $(p,q)\cap A=\emptyset$. Assume that
\Eq{*}{
\frac{f(x)}{g(x)}>\alpha=\frac{f(p)}{g(p)}=\frac{f(q)}{g(q)}\text{ for all }x \in (p,q)\:.
}
The second case, with converse inequality sign, is completely analogous.

Let $\alpha_0:=\sup_{x \in (p,q)} \frac{f(x)}{g(x)}$, and $r \in (p,q)$ be the smallest number with $\alpha_0=\tfrac{f(r)}{g(r)}$. Then

\Eq{*}{
\interior \cl (\Q_+ \cdot \omega M)
&\supseteq \interior \Big\{\big(\alpha y,y\big)\colon \alpha \in (\tfrac fg)(r,q),\,y \in \R_+\Big\}
=\interior \Big\{\big(\alpha y,y\big)\colon \alpha \in (\tfrac fg)(p,q),\,y \in \R_+\Big\}\\
&=\interior \Big\{\big(\tfrac{f(x)}{g(x)}y,y\big)\colon x \in (p,q),\,y \in \R_+\Big\}
=\interior \Big\{\big(f(x)y,g(x)y\big)\colon x \in (p,q),\,y \in \R_+\Big\}.
}

Fix $x_0 \in (p,r)$. Then $(f(x_0),g(x_0)) \in \interior \cl (\Q_+ \cdot  \omega M)$. Consider two cases.

\bigskip

\noindent {\sc Case 1.}
If $M$ is not a line then by Lemma~\ref{lem:Jabl} there exists $k\in \N$ such that $(kf(x_0),kg(x_0)) \in \omega M$. By the definition $ (kf(x_0),kg(x_0)) \in nM$ for some $n \in \N$. In the other words there exists a vector $(v_1,\dots,v_n)$ of elements in $(r,q)$ such that 
\Eq{*}{
k\cdot F(x_0)= k\cdot (f(x_0),g(x_0))= \sum_{i=1}^n (f(v_i),g(v_i))=\sum_{i=1}^n F(v_i)\:.
}
Using the definition of online premean, it implies 
\Eq{*}{
(r,q)\ni \Lv_{F,G}(v_1,\dots,v_n)=\Lv_{F,G}(x_0,\dots,x_0)=x_0 \in (p,r),
}
which lead to a contradiction as $(r,q) \cap (p,r)=\emptyset$.
\bigskip

\noindent {\sc Case 2.} If $M$ is a line then there exist $A,B,C \in \R$ such that 
\Eq{*}{
A f(x) + B g(x)+C=0 \text{ for all } x \in (p,q)\,.\\
}
If $A=0$ or $B=0$ then either $g$ or $f$ is constant on $(p,q)$. It implies that the second function is continuous and strictly monotone, so is the ratio $f/g$. It implies $\tfrac{f(p)}{g(p)} \ne\tfrac{f(q)}{g(q)}$ which lead to a contradiction.

If both $A$ and $B$ are nonzero then: First, both $f$ and $g$ are 1-1 in $(p,q)$ (and therefore strictly monotone). Second,
\Eq{*}{
\frac{f(x)}{g(x)}=\frac{-Bg(x)-C}{g(x)}=-B-\frac C{g(x)} \quad 
\text{ for }x \in (p,q).
}
It implies that $f/g$ restricted to $(p,q)$ is 1-1, i.e $\tfrac{f(p)}{g(p)} \ne \tfrac{f(q)}{g(q)}$ in this case too.
\end{proof}

\begin{lem}\label{lem:posinc}
 Under the assuption of lemma~\ref{lem:pom1} there exists an invertible linear mapping $L \colon \R^2 \to \R^2$ such that for a pair of functions $f_1,g_1\colon I \to \R$ defined by $(f_1,g_1):= L \circ (f,g)$ we have:
 (i) $g_1$ is positive and (ii) $f_1/g_1$ is strictly increasing.
\end{lem}

\begin{proof}
 In view of lemma~\ref{lem:pom1} there exists $\alpha,\beta \in \R$ such that $\alpha f+\beta g$ is nowhere vanishing. Define $p:=1$ if $\alpha f+\beta g$ is positive and $p:=-1$ otherwise.
 
 Take any vector $(\gamma,\delta)$ which is linealrly independent with $(\alpha,\beta)$. Applying Lemma~\ref{lem:pom2} with $f \leftarrow \gamma f+\delta g$ and $g \leftarrow \alpha f+\beta g$ we obtain that $\frac{\gamma f+\delta g}{\alpha f+\beta g}$ is strictly monotone. Take $q:=1$ if it is increasing and $q:=-1$ otherwise. Let $L:=\begin{psmallmatrix}pq\gamma&pq\delta\\p\alpha&p\beta\end{psmallmatrix}$.
 
 Then $\det L=q \cdot \det\begin{psmallmatrix}
                     \gamma & \delta \\ \alpha & \beta 
                    \end{psmallmatrix} \ne 0
$. Furthermore we can easily verify that both (i) and (ii) holds true for a pair $f_1:=pq(\gamma f+\delta g)$ and $g_1:=p(\alpha f+\beta g)$.
\end{proof}

\subsection{Proof of Theorem~\ref{thm:RepInvT2}}
Let $\Lv_{F,G}$ be an arbitrary repetition invariant mean on $I$ of type $\TYPE{2}$ without negligable element, where $F(x)=(f(x),g(x))$ for $f,g\colon I \to \R$ and $G \colon \omega F(I) \to I$.

By Lemma~\ref{lem:posinc} and Remark~\ref{rem:inv} we may assume without loss of generality that $g$ is positive and $f/g$ is strictly increasing, that is Bajraktarevi\'c mean $\B_{f,g}$ is well defined.

Take $n \in \N$, $a \in I^n$, and define $m:=\B_{f,g}(a)$. We shall prove that $\Lv_{F,G}(a)=m$. Let

\Eq{*}{
\theta=\theta(m):=\frac{f(m)}{g(m)}= \frac{f(a_1)+f(a_2)+\dots+f(a_n)}{g(a_1)+g(a_2)+\dots+g(a_n)}.
}
Then
\Eq{*}{
\theta(g(a_1)+g(a_2)+\dots+g(a_n))&=f(a_1)+f(a_2)+\dots+f(a_n),\\
(f(a_1)-\theta g(a_1))+\cdots+(f(a_n)-\theta g(a_n))&=0
}

If now replace $f$ by $f_m:=f-\theta g$ and consider $F_m=(f_m,g)$ (and related function $G_m$ so that $\Lv_{F,G}=\Lv_{F_m,G_m}$) we obtain
\Eq{*}{
\sum_{i=1}^n (f_m(a_i),g(a_i))=\big(0,\sum_{i=1}^n g(a_i) \big)\:.
}
As $f_m(m)=0$ we get $G_m(0,ng(m))=m$ for all $n \in \N$. As $f_m/g$ is strictly increasing we get that $f_m$ changes its sing in a neighbourhood of $m$. In particular for all $\varepsilon>0$ there exists $u,u'\in I$ such that $u<u'$, $s \in (u,u')$,  $f_m(u)=-f_m(u')$, and $u'-u<\varepsilon$. Thus, by the definition of $\Lv$ we have
\Eq{*}{
G_m\big(0,ng(m)+k\big(g(u)+g(u')\big)\big) \in (u,u')\quad \text{ for all }n,\,k \in \N\:.
}

If $f_m=\alpha g$ in a neighbourhood of $m$ then $f=(\alpha+\theta)g$, i.e. the mean restricted to this neighbourhood is of type $\TYPE{1}$ what lead to a contradiction.

Otherwise one can find a pair $(v,v')$ ($u<v<s<v'<u'$) such that $f_m(v)=-f_m(v')$ and 
$g(v)+g(v')\ne g(u)+g(u')$. Then we have
\Eq{E:G_s}{
G_s\big(0,g(s)+k\big(g(v)+g(v')\big)\big) \in (v,v')\subset (u,u')\quad \text{ for all }k \in \N\:.
}

Consider two cases (recall that $m$ is fixed).

\bigskip

\noindent {\sc Case 1.}
$g=C+D f_m$ for some $C, D \in \R$ on some interval $X_m \ni m$. 

Assume that $X_m$ is a maximal interval with this property. Let $\bar G(p,q):=G_m(p,qC+Dp)$, $n \in \N$, and $x =(x_1,\dots,x_n)\in X_m^n$
\Eq{*}{
\Lv_{F_m,G_m}(x)&=G_m\Big(\sum_{i=1}^n f_m(x_i),\sum_{i=1}^n g(x_i)\Big)\\
&=G_m\Big(\sum_{i=1}^n f_m(x_i),nC+\sum_{i=1}^n Df_m(x_i)\Big)=\bar G\Big(\sum_{i=1}^n f_m(x_i),n\Big).
}
which in view of (the proof of) Proposition~\ref{prop:T1+}
implies 
\Eq{*}{\Lv_{F,G}(x)=\Lv_{F_m,G_m}(x)=\QA{f_m}(x)=\B_{f_m,1}(x)=\B_{f_m,g}(x)=\B_{f,g}(x) \qquad (x \in X_m^n).
}
Therefore $G(p,q)=(\tfrac fg)^{-1}(\frac pq)$ for $(p,q) \in \omega F(X_m)$. Since $f_m=f-\theta g$, there exist $p,q \in \R$ such that $g(x)=p+qf(x)$ for all $x \in X_m$.

If $X_m=I$ then, in view of Proposition~\ref{prop:T1+}, we have $\Lv_{F,G}(a)=\QA{f}(a)=\B_{f,g}(a)=m$.

Otherwise we have $\Lv_{F,G}|_{X_m}=\QA{f}|_{X_m}=\B_{f,g}|_{X_m}$.  In particular $G(p,q)=(\tfrac fg)^{-1}(\tfrac pq)$ for all $(p,q) \in \omega F(X_m)$
such that $(\tfrac fg)^{-1}(\tfrac pq)=m$.
Applying this for all $m\in X_m$ we obtain
\Eq{Erh}{
G(p,q)=(\tfrac fg)^{-1}(\tfrac pq)\text{ for all }(p,q) \in \omega F(X_m).
}

As $\Lv_{F,G}$ is repetition invariant we have
\Eq{Eri}{
G(tp,tq)=G(p,q)\text{ for all } t \in \Q,\,(p,q)\in \omega F(I) \text{ with }(tp,tq)\in \omega F(I).
}

Since $F(I)$ is not a segment we obtain by Lemma~\ref{lem:Jabl}, we get 
\Eq{Erj}{
(p,q) \in F(I) \Rightarrow \exists_{M_0>0} \forall_{M>M_0} (Mp,Mq)\in \omega F(I)
}

Thus, as $G$ is continuous, we have 
\Eq{*}{
G(x,y)=\lim_{\substack{t \to \infty\\(tx,ty)\in \omega F(I)}} G(tx,ty)\qquad\text{ for all }(x,y)\in \omega F(I)\:.
}

Binding this property with \eq{Erh}, \eq{Eri}, and \eq{Erj} we have (see also \cite{JarPas19})
\Eq{*}{
G(p,q)=(\tfrac fg)^{-1}(\tfrac pq)\text{ for all } (p,q) \in (\R_+ \cdot \omega F(X_m)) \cap \omega F(I).
}

It implies that $G(p,q)$ equals $(\tfrac fg)^{-1}(\tfrac pq)$ on every semiline determined by some element of $\omega F(X_m)$. Thus 
$G(p,q)=m$  for all $(p,q) \in \omega F(I)$ with $(\tfrac fg)^{-1}(\tfrac pq)=m$. Thus $\Lv_{F,G}(a)=m$.

\noindent {\sc Case 2.} $f$ and $g$ are linearly independent in every neighbourhood of $m$.

There exist two sequences $(v_r)_{r \in \R_+}$ and $(v_r')_{n \in \R_+}$ of points in $I$ such that:
\begin{itemize}
\item The mapping $r \mapsto (v_r,v_r') $ is continuous;
\item $(v_r)$ is increasing, and $(v_r')$ is deceasing;
\item $\lim_{r \to \infty} v_r= \lim_{r \to \infty} v_r'=m$;
  \item $v_r<m<v_r'$ for all $r \in \R_+$;
\item $f(v_r) = -f(v_r')$ for all $r \in \R_+$.
\end{itemize}

Then we obtain that
\Eq{*}{
R:=\{r \in \R_+ \colon g(v_r)+g(v_r')\ne2g(m)\}
}
is a dense subset of $\R_+$.
Define, for each $r \in R$, an open interval
\Eq{*}{
P_r:=\Big(\min\big(g(v_r)+g(v_r'),2g(m)\big),\max\big(g(v_r)+g(v_r'),2g(m)\big)\Big).
}
Fix $r_0 \in R$ arbitrarily. Applying \eq{E:G_s} with $(v,v') \leftarrow (v_r,v_r')$ for all $r \in (r_0,+\infty)$ simultaneously, we obtain
\Eq{*}{
G_m(0,g(m)+k(g(v_r)+g(v_r'))) &\in (v_r,v_r') \subset (v_{r_0},v'_{r_0}) \qquad\qquad &(k\in \N,\,r\in (r_0,+\infty)\:).}
Moreover, as $g$ is continuous we obtain that the mapping $(r_0,\infty) \ni r \mapsto g(v_r)+g(v'_r)$ is onto $P_{r_0}$, thus
\Eq{*}{
G_m(0,g(m)+kx) &\in (v_{r_0},v_{r_0}')  \qquad\qquad &(k\in\N,\,x\in P_{r_0}).
}

But, as $P_{r_0}$ is an open interval and $\sup P_{r_0} \ge2g(m)>0$, there exists $K_{r_0} \in \R$ such that $(K_{r_0},\infty) \subset g(m)+\N \cdot P_{r_0}$. It~implies $G_m(0,z) \in (v_{r_0},v_{r_0}')$ for all $z >K_{r_0}$ and, as an immediate result,
\Eq{*}{
v_{r_0} \le \liminf_{z \to \infty} G_m(0,z) 
\le\limsup_{z \to \infty} G_m(0,z) \le v_{r_0}'.
}
Upon passing a limit $r_0 \to \infty$, $r_0 \in R$ (recall that $R$ is a dense subset of $\R_+$) we obtain
%
\Eq{*}{
\lim_{z \to \infty} G_m(0,z)=m\:.
}

But, in view of repetition invariance of $\Lv_{F_m,G_m}$ on $I$, we have $G_m(kx,ky)=G_m(x,y)$ for all $(x,y) \in \omega F_m(I)$ and $k \in \N$ so
\Eq{*}{
G_m(F_m(a_1)+\dots+F_m(a_n))
=G_m(0,\sum_{i=1}^n g(a_i))
=\lim_{m \to \infty} G_m(0,m\cdot \sum_{i=1}^n g(a_i))
=\lim_{z \to \infty} G_m(0,z)=m.
}
therefore $\Lv_{F_m,G_m}(a)=m$. Now, as $\Lv_{F,G}=\Lv_{F_m,G_m}$, we get 
$\Lv_{F,G}(a)=\Lv_{F_m,G_m}(a)=m$.

\end{document}